\documentclass[11pt]{article}
\usepackage{amsfonts}
\usepackage{amssymb}
\usepackage{amssymb,amsfonts,amsmath,amsthm,cite, color}
\usepackage{epsfig}
\newtheorem{thm}{Theorem}[section]

\newtheorem{lem}{Lemma}[section]
\newtheorem{prop}{Proposition}[section]

\makeatletter \@addtoreset{equation}{section}

\newcommand{\Z}{\mathbb{Z}}
\newcommand{\Q}{\mathbb{Q}}

\newcommand{\C}{\mathbb{C}}

\def\H{\mathbb{H}}

\begin{document}
\begin{center}
{{\Large\bf The number of cubic partitions modulo powers of 5}}

\vskip 6mm

\end{center}

\begin{center}{Xinhua, Xiong\\
             Department of Mathematics, China Three Gorges University, Yichang 443002,
P.R. China \\
xinhuaxiong@ctgu.edu.cn}          
\end{center}

\date{}

\begin{abstract}
The notion of cubic partitions is introduced by
Hei-Chi Chan and named by Byungchan Kim in connection with Ramanujan's cubic continued fractions.
Chan proved that cubic partition function has Ramanujan Type congruences modulo powers of $3$.
In a recent paper, William Y.C. Chen and Bernard L.S. Lin studied the congruent property of the cubic partition function
modulo $5$. In this note, we give Ramanujan type congruences for cubic partition function modulo powers of $5$.
\end{abstract}

\section{Introduction}\label{intro}
Let $p(n)$ denote the number of the unrestricted partitions of $n$, Ramanujan  discovered and
later proved that for every non negative integer $n$, we have:
\begin{eqnarray} \label{ramanujan}
p(5n+4)&\equiv 0\pmod 5,\cr
p(7n+5)&\equiv 0\pmod 7,\cr
p(11n+6)&\equiv 0\pmod{11}.
\end{eqnarray}
Much more is known than (\ref{ramanujan}). In fact, for every integer $\alpha \ge 1$ and every non negative integer $n$:
\begin{equation} \label{r1}
p(5^{\alpha}n + \delta_{5,\alpha}) \equiv 0 \pmod{5^{\alpha}},
\end{equation}
\begin{equation} \label{r2}
p(7^{\alpha}n + \delta_{7,\alpha})\equiv 0 \pmod{7^{[\frac{\alpha +2}{2}]}},
\end{equation}
\begin{equation} \label{r3}
p(11^{\alpha}n + \delta_{11,\alpha}) \equiv 0 \pmod{11^{\alpha}}.
\end{equation}
Here $\delta_{t,\alpha}$ is the reciprocal modulo $t^{\alpha}$ of $24$.
(\ref{r1}) and (\ref{r2}) were first proved by G.N. Watson in 1938, see \cite{Watson38}. Hirschhorn and Hunt \cite{Hirschhorn}
gave an elementary proof of (\ref{r1}). Garvan \cite{Garvan} gave an elementary of (\ref{r2}). (\ref{r3})
was proved by A.O.L. Atkin \cite{Atkin} in 1967. After that, the generalizations of (\ref{r1}), (\ref{r2}) and
(\ref{r3}) for other partition functions have been investigated by many mathematicians. Let $q(n)$
denote the number of partitions of $n$ into distinct parts, Gordon and Hughes \cite{Gordon84} obtained some congruences for
$q(n)$ modulo powers $5$. They proved that for all integers $\alpha \ge 0$ and $n \geq 0$
\begin{equation*} \label{r_01(0)}
q(5^{2\alpha +1}n + \delta_{\alpha}) \equiv 0 \pmod{5^{\alpha}}.
\end{equation*}
Where $\delta_{\alpha}$ is the reciprocal of $-24$ modulo $5^{2\alpha +1}$. Recently, motivated by the
study of Ramanujan's cubic continued fraction, Hei-Chi Chan \cite{Chan08a} \cite{Chan08b} and Byungchan Kim \cite{Kim08} introduced the notion of cubic partition of nonnegative integers. By definition, the generating function of the number of cubic partitions of $n$ is
\begin{equation}\label{Defa}
\sum_{n=0}^\infty a(n)q^n=\prod_{n=1}^{\infty}\frac{1}{(1-q^{n})(1-q^{2n})}.
\end{equation}
From an elegant identity on the Ramanujan's cubic continued fraction, Chan established the
generating function of $a(3n+2)$:
\begin{equation}\label{Defa}
\sum_{n=0}^\infty a(3n+2)q^n=3\prod_{n=1}^{\infty}\frac{(1-q^{3n})^3(1-q^{6n})^3}{(1-q^{n})^4(1-q^{2n})^4},
\end{equation}
which led the following congruence
$$
a(3n+2)\equiv 0 \pmod{3}.
$$
Moreover, He obtained the following Ramanujan type congruences for $a(n)$ similar to (\ref{r1}) (\ref{r2})
(\ref{r3}):
\begin{thm}\label{thm1.1}
For $\alpha\geq 1$,
\begin{equation}\label{genconga}
a(3^{\alpha}n+c_{\alpha})\equiv  0\ ({\rm mod}\ 3^{\alpha+\delta(\alpha)}),
\end{equation}
where $c_{\alpha}$ is the reciprocal modulo $3^{\alpha}$ of $8$, and $\delta(\alpha)=1$
if $\alpha$ is even and $\delta(\alpha)=0$ if $\alpha$ is odd.
\end{thm}
In a recent paper \cite{Chen and Lin 09}, William Y.C. Chen and Bernard L.S. Lin studied the congruence of $a(n)$
modulo $5$, they obtained that
\begin{thm}\label{thm1.2} For every nonnegative integer $n$, we have
\begin{eqnarray}\label{mod5}
a(25n+22)&\equiv& 0\ ({\rm mod}\ 5).
\end{eqnarray}
\end{thm}
The aim of this note is the generalization of Theorem \ref{thm1.2} to the congruences of $a(n)$ modulo
all powers of $5$, which is analogues to Ramanujan type congruences for $a(n)$ modulo powers of $3$.
The main result of the paper is the following theorem:
\begin{thm}\label{thm1.4}
If both $\alpha$ and $n$ are non negative integers, then
\begin{equation}
a(5^{2\alpha+2}n+\delta_{\alpha})\equiv  0\ ({\rm mod}\ 5^{\alpha+1}),
\end{equation}
where $\delta_{\alpha}$ is the reciprocal modulo $5^{2\alpha+2}$ of $8$.
\end{thm}
The paper is organized as follows. In Section 2 we recall some properties of modular forms, especially on
Eta-products. In Section 3, we give a different proof of Theorem \ref{thm1.2}. The proof of Theorem \ref{thm1.4} is given in Section 3. the appendix contains some data obtained by Wolfram Mathematica 6.
\section{Preliminaries}\label{preli}
Let $\mathbb{H}:=\{z\in\mathbb{C}|{\rm Im}(z)>0\}$ denote the upper half of the complex plane, and $SL_2(\mathbb{Z})$
be the full modular group, for a positive integer $N$, we define the subgroup $\Gamma_0(N)$ of $SL_2(\mathbb{Z})$ as follows:
$$
\Gamma_0(N):=\left\{\left(
                 \begin{array}{cc}
                   a & b \\
                   c & d \\
                 \end{array}
               \right)\Big|c\equiv 0\ ({\rm mod}\ N)
\right\}.
$$
Any $\gamma=\left(
                 \begin{array}{cc}
                   a & b \\
                   c & d \\
                 \end{array}
               \right)\in SL_2(\mathbb{Z})$ acts on the upper
half  of complex plane by the linear
fractional transformation
$$
\gamma z:=\frac{az+b}{cz+d}.
$$
Let $f(z)$ be a function on $\H$  satisfying
$
f(\gamma z)=f(z),
$
if $f(z)$ is meromorphic on $\H$ and at  all the cusps of $\Gamma_0(N)$, then we call $f(z)$
a meromorphic modular function with respect to $\Gamma_0(N)$. The set of all such functions is denoted by
$\mathcal{M}_0(\Gamma_0(N))$.
Dedekind's eta function is defined by
$$
\eta(z):=q^{\frac{1}{24}}\prod_{n=1}^\infty (1-q^n),
$$
where $q=e^{2\pi iz}$ and ${\rm Im}(z)>0$. It is well-known that
$\eta(z)$ is holomorphic and does not vanish on $\mathbb{H}$.

A function $f(z)$ is called an eta-product if it can
be written in the form of
$$
f(z)=\prod_{\delta |N}\eta^{r_{\delta}}(\delta z),
$$
where $N$ and $\delta$ are  natural numbers and  $r_{\delta}$ is an integer. The following
fact which is  due to Gordon-Hughes \cite{Gordon84} and Newman \cite{Newman} is useful  to verify  whether an eta-product is a
modular function.
\begin{prop}[\cite{cmbs}]\label{prop2.1}
If $f(z)=\prod_{\delta|N}\eta^{r_{\delta}}(\delta
z)$ is an eta-product with
$$
k=\frac{1}{2}\sum_{\delta|N}r_{\delta}\in \mathbb{Z},
$$
satisfying the following conditions:
\begin{equation}\label{con1}
\sum_{\delta|N}\delta r_{\delta}\equiv 0 \ ({\rm mod}\ 24)
\end{equation}
and
\begin{equation}\label{con2}
\sum_{\delta|N}\frac{N}{\delta} r_{\delta}\equiv 0 \ ({\rm mod}\
24),
\end{equation}
then $f(z)$ satisfies
\begin{equation}\label{relation1}
f\left(\frac{az+b}{cz+d}\right)=\chi(d)(cz+d)^kf(z)
\end{equation}
for each $\left(
           \begin{array}{cc}
             a & b \\
             c & d \\
           \end{array}
         \right)\in \Gamma_0(N)$.
          Here the character $\chi$ is
         defined by $\chi(d):=\left(\frac{(-1)^ks}{d}\right)$, where
         \[ s:=\prod_{\delta|N}\delta^{r_{\delta}}\] and
         $\left(\frac{m}{n}\right)$ is Kronecker symbol.
\end{prop}In particular, an eta-product is in $\mathcal{M}_0(\Gamma_0(N))$ if $k=0$ and $s$
is a square of rational number.
The following property which is due to Ligozat gives the analytic orders of an eta-product at the
cusps of $\Gamma_0(N)$.
\begin{prop}[\cite{cmbs}]\label{prop2.2}
Let $c,d$ and $N$ be positive integers with $d|N$ and $(c,d)=1$. If
$f(z)$ is an eta-product satisfying the conditions in Proposition
\ref{prop2.1} for $N$, then the order of vanishing of $f(z)$ at the
cusp $\frac{c}{d}$ is
\begin{equation}\label{formula}
\frac{N}{24}\sum_{\delta
|N}\frac{(d,\delta)^2r_{\delta}}{(d,\frac{N}{d})d\delta}.
\end{equation}
\end{prop}
Let $p$ be a prime and
$
f(q)= \sum_{n\ge n_0}^{\infty}a(n)q^n
$
be a formal power series, we define
$
U_p(f(q))=\sum_{pn\ge n_0}a(pn)q^n.
$
\noindent We write $U$ instead of $U_5$ from Section 3. If $f(z)$ is in $\mathcal{M}_0(\Gamma_0(N))$,
then $f(z)$ has an expansion at the point $i\infty$ of the form $f(z)=\sum_{n=n_0}^{\infty}a(n)q^n$
where $q=e^{2\pi iz}$ and ${\rm Im}(z)>0$. We call this expansion the Fourier series
of $f(z)$. Moreover We define $U_p(f(z))$ to be the result of applying $U_p$
to the Fourier series $f(z)$.
We use two results on the $U$-operator acting on space $\mathcal{M}_0(\Gamma_0(N))$
stated by Gordon and Hughes \cite{Gordon84}. One is if $f(z) \in \mathcal{M}_0(\Gamma_0(pN))$, where $p|N$, then
$U_p(f(z)) \in \mathcal{M}_0(\Gamma_0(N))$. The other result gives the lower bounds of orders
of $U_5(f(z))$ at the cusps of $\Gamma_0(10)$ in terms of orders of $f(z)$ at the cusps of $\Gamma_0(50)$.
In this case, $\Gamma_0(10)$ has $4$ cusps, represented by $0, \frac{1}{2}, \frac{1}{5}, \frac{1}{10}(=i\infty).$
$\Gamma_0(50)$ has $12$ cusps, represented by $0, \frac{1}{2}, \frac{1}{5}, \frac{2}{5}, \frac{3}{5},
\frac{4}{5}, \frac{1}{10}, \frac{3}{10}, \frac{7}{10}, \frac{9}{10},\frac{1}{25}, \frac{1}{50}(=i\infty).$ By Ligozat's
formula on the analytic orders of an eta-product, if $f(z)$ is an eta-product in $\mathcal{M}_0(\Gamma_0(N))$, then $f(z)$ has the same order at cusps which have the same denominators. The order of $U_5(f(z))$ at a cusp $r$ of $\Gamma_0(10)$
is denoted by $ord_r U(f)$, and the order of $f(z)$ at a cusp of $s$ of $\Gamma_0(50)$ is denoted by $ord_s f$.
\begin{prop}[\cite{Gordon84}]\label{prop2.3}
Let $f(z)$ be an eta-product in $\mathcal{M}_0(\Gamma_0(50))$, then $U_5(f(z))$ in $\mathcal{M}_0(\Gamma_0(10))$, and
\begin{eqnarray*}
&& ord_0 U(f) \geq\text{min}\,(ord_0f, ord_{\frac{1}{5}}f), \quad
ord_{\frac{1}{2}} U(f) \geq\text{min}\,(ord_{\frac{1}{2}}f, ord_{\frac{1}{10}}f),\cr
&& ord_{\frac{1}{5}} U(f) \geq\frac{1}{5}ord_{\frac{1}{25}}f, \quad\quad\quad\quad\quad
ord_{\frac{1}{10}} U(f) \geq \frac{1}{5}ord_{\frac{1}{50}}f.
\end{eqnarray*}
Moreover, $U(f)$ has no poles on $\H$ except the cusps.
\end{prop}
\section{Cubic partition modulo $5$}\label{sec3}
In this section, we give a different proof of Theorem \ref{thm1.2} which was proved in \cite{Chen and Lin 09} by using Sturm's theorem and machinery verification. Our method is similar to the proofs of the partition identities of the generating functions of $p(5n+4),\, p(7n+5)$ and $p(11n+6)$
in the Section $3$ of Chan-Lewis \cite{H-Chan}.
Define an eta-product
\begin{equation*}
F:=F(z)=\frac{\eta(25z)}{\eta(z)}\frac{\eta(50z)}{\eta(2z)},
\end{equation*}
\noindent setting $N=50$, we find $F(z)$ satisfies the conditions of Newman-Gordon-Hughes's theorem i.e. Proposition \ref{prop2.1},
so $F(z)$ is in $\mathcal{M}_0(\Gamma_0(50))$. We use Ligozat's formula (\ref{formula}) to calculate the orders of $F(z)$
at the cusps $\frac{c}{d}$, for $ d=1,\,2,\,5,\,10\,$. We give the calculation of the case of  $d=1$ as an example as follows:
\begin{eqnarray*}
ord_0F&=&\frac{50}{24\times (1, \frac{50}{1})} \sum_{\delta | 50}\frac{(1, \delta)^2}{\delta}r_{\delta}\\
&=&\frac{50}{24}\times \left(\frac{(1, 25)^2}{25}\times 1 +\frac{(1, 50)^2}{50}\times 1 +\frac{(1, 1)^2}{1}\times (-1)+\frac{(1, 2)^2}{2}\times (-1) \right)\\
&=& -3.
\end{eqnarray*} Similar calculations give
\begin{center}\label{tab1}
\begin{tabular}{|c|c|c|c|c|c|c|}
  \hline
  $d$ & 1 & 2 & 5 & 10 & 25 & 50 \\
  \hline
  $ord_{c/d}F$
 & $-3$ & $-3$ & 0 & 0 & 3 & 3 \\
  \hline
\end{tabular}
\end{center}
\noindent By Proposition \ref{prop2.3}, the orders of $U(F)$ at the cusps of $\Gamma_0(10)$ satisfy
\begin{center}\begin{tabular}{|c|c|c|c|c|}
  \hline
 $d$ & 1 & 2 &5 & 10 \\
  \hline
   $ord_{c/d}U_5(F)\geq$ &$ -3$ & $-3$ & 1 & 1 \\
  \hline
\end{tabular}
\end{center}
and $U(F)$ is holomorphic on $\H$. We note that the poles of $U(F)$ only appear at the cusps $0$ and $\frac{1}{2}$.
We define another eta-product
$$
A:=A(z)=\frac{\eta^2(5z)}{\eta^2(z)}\frac{\eta^2(10z)}{\eta^2(2z)}.
$$
By Proposition \ref{prop2.1}, we find that $A$ is in $\mathcal{M}_0(\Gamma_0(10))$. Ligozat's formula on an order of a cusp $c/d$ for an eta-product gives
\begin{center}\begin{tabular}{|c|c|c|c|c|}
  \hline
 $d$ & 1 & 2 &5 & 10 \\
  \hline
   $ord_{c/d}A$ & $-1$ & $-1$ & 1 & 1 \\
  \hline
\end{tabular}
\end{center}
and $A$ is holomorphic and non-zero elsewhere. Since the Riemann surface $(\H\cup \Q \cup{i\infty}) /\Gamma_0(10)$
has genus $0$, $\mathcal{M}_0(\Gamma_0(10))$ has one generator as a field. The orders of $A$ show that $U(F)$ is a polynomial in $A$ of degree at most $3$. Thus we can suppose that
$$
U(F)= c_0 + c_1A + C_2 A^2 +c_3 A^3, \quad c_i(i=1,2,3) \in \C
$$
Since
\begin{eqnarray*}
F(z)&=& q^3\prod_{n=1}^{\infty}\frac{(1-q^{25n})(1-q^{50n})}{(1-q^n)(1-q^{2n})}\\
&=& q^3 + q^4 +3q^5 +4q^6 +9q^7 + 128q^8 +23q^9+\cdot\cdot\cdot\\
\end{eqnarray*}
And
\begin{eqnarray*}
A&=& q\prod_{n=1}^{\infty}\frac{(1-q^{5n})^2(1-q^{10n})^2}{(1-q^n)^2(1-q^{2n})^2}\\
&=& q+2q^2+ 7q^3 + 14q^4 +35q^5 +64q^6 +\cdot\cdot\cdot\\
\end{eqnarray*}
The comparison of the Fourier coefficients of $A$ and $U(F)$ shows that
\begin{equation}\label{U(F)}
U(F)=3A +25A^2 +125 A^3.
\end{equation}
Now if $A$ is regarded as a modular function in $\mathcal{M}_0(\Gamma_0(50))$, Ligozat's formula
shows that the orders of $A$ at the cusps $c/d$ of $\Gamma_0(50)$ are:
\begin{center}\label{tab4}
\begin{tabular}{|c|c|c|c|c|c|c|}
  \hline
  $d$ & 1 & 2 & 5 & 10 & 25 & 50 \\
  \hline
  $ord_{c/d}A$
 & $-5$ & $-5$ & 1 & 1 & 1 & 1 \\
  \hline
\end{tabular}
\end{center}
Hence by Proposition \ref{prop2.3}, we obtain the following lower bounds for the order of $U(A)$ at the cusps
of $\Gamma_0(10)$:
\begin{center}\begin{tabular}{|c|c|c|c|c|}
  \hline
 $d$ & 1 & 2 &5 & 10 \\
  \hline
   $ord_{c/d}U_5(A)\geq$ & $-5$ & $-5$ & 1 & 1 \\
  \hline
\end{tabular}
\end{center}
The same reasoning gives that $U(A)$ is a polynomial in $A$ of degree at most $5$. The comparison of the
Fourier coefficients of $U(A)$ and $A^i, 1\leq i \leq 5$ shows that
\begin{equation}\label{U(A)}
U(A)=35 A + 700 A^2 + 6875 A^3 + 31250 A^4 + 78125 A^5:=35A+25R(A),
\end{equation}where $25R(A)=700 A^2 + 6875 A^3 + 31250 A^4 + 78125 A^5$.
Generally, by the orders of $F$, $A$ and Proposition \ref{prop2.3} , we have the lower bounds for the orders of $U(FA^i),\,i\geq 1$,and $U(A^i),\, i\geq1$
at the cusps $c/d$of $\Gamma_0(10)$:
\begin{center}\begin{tabular}{|c|c|c|c|c|}
  \hline
 $d$ & 1 & 2 &5 & 10 \\
  \hline
   $ord_{c/d}U_5(A^i)\geq$ & $-5i$ & $-5i$ & $i/5$ & $i/5$ \\
  \hline
\end{tabular}
\end{center}
\begin{center}\begin{tabular}{|c|c|c|c|c|}
  \hline
 $d$ & 1 & 2 &5 & 10 \\
  \hline
   $ord_{c/d}U_5(FA^i)\geq$ & $-3-5i$ & $-3-5i$ & $(3+i)/5$ & $(3+i)/5$ \\
  \hline
\end{tabular}
\end{center}
So $U(A^i)$ is a polynomial in $A$ of degree at most $5i$, $U(FA^i)$ is a polynomial in $A$ of degree
at most $3+5i$.
For $i \geq 1$, we can write
\begin{eqnarray*}
U(A^i)= \sum_{j\geq 0}a_{ij}A^j,\\
U(FA^i)=\sum_{j\geq 0}b_{ij}A^j.
\end{eqnarray*}
Where $a_{ij}$ and $b_{ij}$ are complex numbers. By considering the lower bounds for the orders of
$U(A^i)$ and $U(FA^i)$ at the cusps $\frac{1}{5}$ and $0$, we see that $a_{i0}=0, b_{i0}=0$ for all $i\geq 1$, and
$a_{ij}=0$ unless $\frac{i}{5}\leq j \leq 5i$; $b_{ij}=0$ unless $\frac{i}{5}\leq j \leq 5i+3$.

In order to obtain the information on $a_{ij}$ and $b_{ij}$, we search the recurrence of $U(A^i)$ and $U(FA^i)$.
Since $5U(A^i)= \sum_{t=0}^{4}A(\frac{z+t}{5})^i$ is a polynomial in $A$ of degree at most $5i$ for $i\geq 0$.
So the power sums $5U(A^i)$ of $A(\frac{z+t}{5})^i (0\leq t \leq 4)$ are in $\C[A]$. By Newton's identities their elementary symmetric function $\sigma_i$ are
also in $\C[A]$. We let
$$
x_i=A(\frac{z+i-1}{5})^{-1},
$$
then $x_i, 1\leq i\leq 5$ are the roots of the equation
\begin{equation}\label{equation1}
x^5 -\frac{\sigma_4}{\sigma_5}x^4 + \frac{\sigma_3}{\sigma_5}x^3 - \frac{\sigma_2}{\sigma_5}x^2 + \frac{\sigma_1}{\sigma_5}x-
\frac{1}{\sigma_5}=0.
\end{equation}
We can compute $U(A^i),\, -4\leq i \leq 0$ and $U(FA^i), \,-4\leq i \leq 0$ as follows:
$$U(A^0)=1,$$
$$U(A^{-1})= -2-5A,$$
$$U(A^{-2})= -2-125A^2,$$
$$U(A^{-3})= 46-3125A^3,$$
$$U(A^{-4})= -210-78125A^4.$$
From the relations between power sums and the elementary symmetric functions and the equation (\ref{equation1}), we get the following equations:
\begin{eqnarray*}\label{A-inverse}
5U(A^{-1})&=&\sum x_i = \frac{\sigma_4}{\sigma_5},\\
5U(A^{-2})&=& \sum x_i^2 = (\sum x_i)^2 -2\sum x_ix_j\\
&=& \frac{\sigma_4^2}{\sigma_5^2}-2\frac{\sigma_3}{\sigma_5},\\
5U(A^{-3})&=& \sum x_i^3 = (\sum x_i)^3-3(\sum x_i)(\sum x_ix_j) + 3\sum x_ix_jx_k\\
&=& \frac{\sigma_4^3}{\sigma_5^3}-3\frac{\sigma_4}{\sigma_5}\frac{\sigma_3}{\sigma_5}+\frac{\sigma_2}{\sigma_5},\\
5U(A^{-4})&=& \sum x_i^4 =(\sum x_i)^4-4(\sum x_i)(\sum x_ix_j)+2(\sum x_ix_j)^2\\
&+& (\sum x_i)(\sum x_ix_jx_k)-4\sum x_ix_jx_kx_l\\
&=& \frac{\sigma_4^4}{\sigma_5^4}-4\frac{\sigma_4}{\sigma_5}\frac{\sigma_3}{\sigma_5}+\frac{\sigma_3^2}{\sigma_5^2}+
\frac{\sigma_4}{\sigma_5}\frac{\sigma_2}{\sigma_5}-4\frac{\sigma_1}{\sigma_5}.
\end{eqnarray*}
Clearly, these and $$5U(A)=\frac{1}{x_1}+\frac{1}{x_2}+\frac{1}{x_3}+\frac{1}{x_4}+\frac{1}{x_5}$$ determine the $\sigma_i$.
The result is the following:
\begin{eqnarray*}\label{sigma}
\sigma_1 &=& 175A + 3500 A^2 + 34375 A^3 + 156250 A^4 + 390625 A^5,\\
\sigma_2 &=& -140A - 1375 A^2 - 6250 A^3 - 15625 A^4,\\
\sigma_3 &=& 55A+250A^2+625A^3,\\
\sigma_4 &=& -10A-25A^2,\\
\sigma_5 &=& A.
\end{eqnarray*}
From the Newton recurrence for power sums, we have for all $i\geq 1$,
\begin{equation}\label{recurrence}
U(A^i)=\sigma_1 U(A^{i-1})-\sigma_2 U(A^{i-2}) +\sigma_3 U(A^{i-3})-\sigma_4 U(A^{i-4})+\sigma_5 U(A^{i-5}).
\end{equation}
Note the coefficients of $\sigma_i$ above and the initial values of $U(A^i)$  are all in $\Z$,
it follows from \ref{recurrence} that for all $i\geq 1$,
\begin{equation}\label{a_{ij}}
U(A^i)=\sum_{j\geq 1}a_{ij}A^j,\quad a_{ij}\in \Z.
\end{equation}

The functions $U(FA^i)$ satisfy the same recurrence (\ref{recurrence}) as $U(A^i)$. The initial values of $U(FA^i)$ for $-4 \leq i \leq 0$ are
$$U(FA^0)=3 A + 25 A^2 + 125 A^3,$$
$$U(FA^{-1}) = A,$$
$$U(FA^{-2}) = 25A^2,$$
$$U(FA^{-3}) = -75A-625A^2-2500A^3,$$
$$U(FA^{-4}) = -7 +525A +4375A^2 +21875A^3 +15625A^4.$$
So we can deduce that $b_{ij} \in \Z,$ for $ i\geq 1, j\geq 1.$
Now we prove Theorem \ref{thm1.2}. \begin{proof}
We can write
\begin{eqnarray*}
F(z)&=&\frac{\eta(25z)}{\eta(z)}\frac{\eta(50z)}{\eta(2z)}\cr
&=& \left(\sum_{n\geq 3}^{\infty}a(n-3)q^n\right)\prod_{n=1}^{\infty}(1-q^{25n})(1-q^{50n}).
\end{eqnarray*}
Applying $U$-operator($U=U_5$ in the following) on both sides above, by (\ref{U(F)}) we have
\begin{eqnarray}\label{U-1}
U(F)=3A +25A^2 +125 A^3&=&\left(\sum_{5n\geq 3}^{\infty}a(5n-3)q^n\right)\prod_{n=1}^{\infty}(1-q^{5n})(1-q^{10n})\cr
&=& \left(\sum_{n\geq 1}^{\infty}a(5n-3)q^n\right)\prod_{n=1}^{\infty}(1-q^{5n})(1-q^{10n}).
\end{eqnarray}
Putting $$A=q\prod_{n=1}^{\infty}\frac{(1-q^{5n})^2(1-q^{10n})^2}{(1-q^n)^2(1-q^{2n})^2}$$
into (\ref{U-1}), we obtain that
\begin{eqnarray*}
\sum_{n\geq 1}^{\infty}a(5n-3)q^n&=&3q\prod_{n=1}^{\infty}\frac{(1-q^{5n})(1-q^{10n})}{(1-q^n)^2(1-q^{2n})^2}+
25q^2\prod_{n=1}^{\infty}\frac{(1-q^{5n})^3(1-q^{10n})^3}{(1-q^n)^4(1-q^{2n})^4}\cr
&+&
125q^3\prod_{n=1}^{\infty}\frac{(1-q^{5n})^5(1-q^{10n})^5}{(1-q^n)^6(1-q^{2n})^6}.
\end{eqnarray*}
Apply $U$-operator again on both sides of (\ref{U-1}), we obtain that
\begin{equation}\label{U^2F}
U(3 A + 25 A^2 + 125 A^3)=\left(\sum_{5n \geq 1}^{\infty}a(25n-3)q^n\right)\left(\prod_{n=1}^{\infty}(1-q^{n})(1-q^{2n})\right).
\end{equation}
From (\ref{U(A)}), $U(A)=35A +25R(A)$, so
$$
U(3 A + 25 A^2 + 125 A^3)=105A + 75R(A) +25 U(A^2) +125U(A^3),
$$
where $R(A),\, U(A^2),\, U(A^3)$ are in $\Z[A]$ by (\ref{a_{ij}}). We find
\begin{equation}\label{qqq}
\left(\sum_{5n\geq1}^{\infty}a(25n-3)q^n\right)\prod_{n=1}^{\infty}(1-q^{n})(1-q^{2n})\equiv 0 \pmod{5}.
\end{equation}
But
$$
\prod_{n=1}^{\infty}(1-q^{n})(1-q^{2n})\not\equiv 0 \pmod{5},
$$
so (\ref{qqq}) implies that
$$
\sum_{5n\geq 1}^{\infty}a(25n-3)q^n =\sum_{n=0}^{\infty}a(25n+22)q^n \equiv 0 \pmod{5},
$$
which is the Theorem \ref{thm1.2}.
\end{proof}
\section{ Proof of Theorem \ref{thm1.4}}\label{sec4}
In this section, we prove the Theorem \ref{thm1.4}. We first note that for $\alpha \geq 0$,
\begin{eqnarray*}
\sum_{n=0}^{\infty}a(5^{2\alpha+2}n+\delta_{\alpha})q^n&\equiv& \sum_{n=0}^{\infty}a(5^{2\alpha+2}n-(-\delta_{\alpha}))q^n \cr
&:=& \sum_{n=1}^{\infty}a(5^{2\alpha+2}n-{\delta_{\alpha}}^{\prime})q^n.
\end{eqnarray*}
Where ${\delta_{\alpha}}^{\prime}\equiv -\delta_{\alpha} \pmod{5^{2\alpha+2}}$. By the definition of $\delta_{\alpha}$, We find that
\begin{eqnarray*}
{\delta_{\alpha}}^{\prime}&=& 5^{2\alpha+2}-\delta_{\alpha}=\frac{5^{2\alpha+2 }-1}{8}\cr
&=& \frac{25(5^{2\alpha}-1)+24}{8}=25{\delta_{\alpha}}^{\prime}+3,
\end{eqnarray*}
i.e. ${\delta_{0}^{\prime}}=3$, ${\delta_{1}^{\prime}}=78\dots$. By Induction on $\alpha$, we find that
\begin{equation}\label{A}
{\delta_{\alpha}}^{\prime}={\delta_{\alpha-1}}^{\prime}+ 3\times 25^{\alpha}.
\end{equation}
Define $W_1=U(F)$, $W_2=U(W_1)$, in general, for $\alpha\geq 1$, $W_{2\alpha +1}=U(W_{2\alpha}F),
W_{2\alpha+2}=U(W_{2\alpha+1})$. As before, we find that
$$
W_2=U(U(F))=\left(\sum_{n=1}^{\infty}a(25n-{\delta_{0}^{\prime}})q^n\right)\prod_{n=1}^{\infty}(1-q^{n})(1-q^{2n}).
$$
If we suppose for $\alpha \geq 1$
$$
W_{2\alpha}=\left(\sum_{n=1}^{\infty}a(5^{2\alpha}n-{\delta_{\alpha-1}^{\prime}})q^n\right)\prod_{n=1}^{\infty}(1-q^{n})(1-q^{2n}),
$$
then
\begin{eqnarray*}
W_{2\alpha+2}&=& U(U(W_{2\alpha}F))\\
&=& U\left(U\left(\sum_{n=1}^{\infty}a(5^{2\alpha}n-{\delta_{\alpha-1}^{\prime}})q^{n+3}\prod_{n=1}^{\infty}(1-q^{25n})(1-q^{50n})\right)\right)\cr
&=& U\left(U\left(\sum_{n=1}^{\infty}a(5^{2\alpha}n-(\delta_{\alpha-1}^{\prime}+3\times 5^{2\alpha}))q^{n}\prod_{n=1}^{\infty}(1-q^{25n})(1-q^{50n})\right)\right)\cr
&=& \sum_{n=1}^{\infty}a(5^{2\alpha+2}n-{\delta_{\alpha}^{\prime}})q^n \prod_{n=1}^{\infty}(1-q^{n})(1-q^{2n}),
\end{eqnarray*} by using (\ref{A}).
So the Theorem \ref{thm1.4} is equivalent to the congruences
\begin{equation}\label{thm1.5}
W_{2\alpha+2} \equiv 0 \pmod{5^{\alpha+1}}
\end{equation}
hold for every $\alpha \geq 0$.

From Section \ref{sec3}, we know for $i \geq 1$,
\begin{eqnarray*}
U(A^i)= \sum_{j\geq 0}a_{ij}A^j, \, a_{ij}\in \Z,\quad
U(FA^i)=\sum_{j\geq 0}b_{ij}A^j,\, b_{ij} \in \Z.
\end{eqnarray*}
We write the matrices $a=(a_{ij}), b=(b_{ij}),\, i\geq1, j\geq 1$. In base of $A,\,A^2,\,A^3,\,A^4\dots$,
\begin{eqnarray*}
W_1&=&U(FA^0)=(3,25,125,0,0,0,\dots),\cr
W_2&=&U(W_1) =(3,25,125,0,0,0,\dots)a,\cr
W_3 &=& U(W_2F)=(3,25,125,0,0,0,\dots)ab.
\end{eqnarray*}
In general, for $\alpha \geq 1$,
\begin{eqnarray*}
W_{2\alpha+1}&=&(3,25,125,0,0,0,\dots)(ab)^{\alpha},\cr
W_{2\alpha+2}&=&(3,25,125,0,0,0,\dots)(ab)^{\alpha}a.
\end{eqnarray*}
Write $W_k=\left(w_1^{(k)}, w_2^{(k)}, w_3^{(k)}, \dots\right)$, then (\ref{thm1.5}) is equivalent to
\begin{equation}\label{equ4.3}
w_j^{(2\alpha+2)} \equiv 0 \pmod{5^{\alpha+1}}
\end{equation}
for $j \geq 1$ and $\alpha \geq0$.
In order to prove (\ref{equ4.3}), we need the following lemmas.
From the Newton recurrence (\ref{recurrence}) for $U(A^i)$, we know that for $i\geq 5,\,
j\geq 5$,
\begin{eqnarray*}
a_{ij}&=& 175\,a_{i-1,j-1}+3500\,a_{i-1,j-2}+34375\,a_{i-1,j-3}+156250\,a_{i-1,j-4}\\
&+& 390625\,a_{i-1,j-5} +140\,a_{i-2,j-1}+375\,a_{i-2,j-2}+6250\,a_{i-2,j-3}\\
&+& 15625\,a_{i-2,j-4}+55\,a_{i-3,j-1}+250\,a_{i-3,j-2}+625\,a_{i-3,j-3}\\
&+& 10\,a_{i-4,j-1} +25\,a_{i-4,-j-2}+ a_{i-5,j-1}.
\end{eqnarray*}
If we denote the $5$-adic order of integer $m$ by $\pi(m)$, we see that
\begin{eqnarray}\label{star}
\pi(a_{ij})&\geq& \text{min}\,(\pi(a_{i-1,j-1})+2,\, \pi(a_{i-1,j-2})+3,\, \pi(a_{i-1,j-3})+5\cr
&&\pi(a_{i-1,j-4})+7,\pi(a_{i-1,j-5})+8,\,\pi(a_{i-2,j-1})+1\cr
&&\pi(a_{i-2,j-2})+3,\,\pi(a_{i-2,j-3})+5,\,\pi(a_{i-2,j-4})+6,\cr
&&\pi(a_{i-3,j-1})+1,\,\pi(a_{i-3,j-2})+3,\,\pi(a_{i-3,j-3})+4\cr
&&\pi(a_{i-4,j-1})+1,\,\pi(a_{i-4,j-2})+2,\cr
&&\pi(a_{i-5,j-1})).
\end{eqnarray}
\begin{lem}\label{lem4.1}
For all $i,j \geq 1$, we have
\begin{equation*}
\pi(a_{ij})\geq \left[\frac{3j-i}{2} \right].
\end{equation*}
\end{lem}
\begin{proof}
We give the values of $\pi(a_{ij})$ for $1\leq i \leq 5$ which show the inequality holds for $1\leq i\leq 5$.
\begin{eqnarray*}
\pi(a_{1,1})=1,\, \pi(a_{1,2})=2,\, \pi(a_{1,3})=4,\, \pi(a_{1,4})=6,\, \pi(a_{1,5})=7,\cr
\pi(a_{2,1})=0,\, \pi(a_{2,2})=2,\, \pi(a_{2,3})=4,\,  \pi(a_{2,4})=5,\,\pi(a_{2,5})=7,\pi(a_{2,6})=9,\,\cr
\pi(a_{2,7})=10,\, \pi(a_{2,8})=12,\, \pi(a_{2,9})=14,\, \pi(a_{2,10})=15,\cr
\pi(a_{3,1})=0,\,\pi(a_{3,2})=2,\,\pi(a_{3,3})=4,\,\pi(a_{3,4})=5,\,\pi(a_{3,5})=7,\,\pi(a_{3,6})=9,\cr
\pi(a_{3,7})=10,\,\pi(a_{3,8})=12,\,\pi(a_{3,9})=15,\,\pi(a_{3,10})=15,\,\pi(a_{3,31})=18,\cr
\pi(a_{3,12})=18,\,\pi(a_{3,13})=20,\,\pi(a_{3,14})=22,\,\pi(a_{3,15})=23,\cr
\pi(a_{4,1})=0,\,\pi(a_{4,2})=1,\,\pi(a_{4,3})=3,\,\pi(a_{4,4})=5,\,\pi(a_{4,5})=6,\pi(a_{4,6})=9,\,\cr
\pi(a_{4,7})=10,\,\pi(a_{4,8})=13,\,\pi(a_{4,9})=14,\,\pi(a_{4,10})=15,\pi(a_{4,11})=16,\,\cr
\pi(a_{4,12})=18,\,\pi(a_{4,13})=20,\,\pi(a_{4,14})=20,\,\pi(a_{4,15})=23,\pi(a_{4,16})=25,\,\cr
\pi(a_{4,17})=26,\,\pi(a_{4,18})=28,\,\pi(a_{4,19})=30,\,\pi(a_{4,20})=31,\pi(a_{5,1})=0,\,\cr
\pi(a_{5,2})=2,\,\pi(a_{5,3})=4,\,\pi(a_{5,4})=5,\,\pi(a_{5,5})=7,\,\pi(a_{5,6})=9,\pi(a_{5,7})=10,\,\cr
\pi(a_{5,8})=12,\,\pi(a_{5,9})=14,\,\pi(a_{5,10})=14,\,\pi(a_{5,11})=18,\,\pi(a_{5,12})=18,\cr
\pi(a_{5,13})=19,\,\pi(a_{5,14})=21,\,\pi(a_{5,15})=25,\,\pi(a_{5,16})=24,\,\pi(a_{5,17})=26,\,\cr
\pi(a_{5,18})=28,\pi(a_{5,19})=29,\,\pi(a_{5,20})=32,\,\pi(a_{5,21})=34,\,\cr
\pi(a_{5,22})=35,\,\pi(a_{5,23})=37,\,\pi(a_{5,24})=39, \pi(a_{5,25})=39.
\end{eqnarray*}
For $i\geq 6$, using induction assumption  on $i-1$, we find that
\begin{eqnarray*}
\pi(a_{i-1,j-1})+2 &\geq& \left[\frac{3(j-1)-(i-1)}{2} \right] +2=\left[ \frac{3j-i+2}{2}\right]\geq\left[\frac{3j-i}{2} \right],\cr
\pi(a_{i-1,j-2})+2 &\geq& \left[\frac{3(j-2)-(i-1)}{2} \right] +2=\left[ \frac{3j-i+1}{2}\right]\geq\left[\frac{3j-i}{2} \right],\cr
\pi(a_{i-1,j-3})+5 &\geq& \left[\frac{3(j-3)-(i-1)}{2} \right] +5=\left[ \frac{3j-i+2}{2}\right]\geq\left[\frac{3j-i}{2} \right],\cr
\pi(a_{i-1,j-4})+7 &\geq& \left[\frac{3(j-4)-(i-1)}{2} \right] +7=\left[ \frac{3j-i+3}{2}\right]\geq\left[\frac{3j-i}{2} \right],\cr
\pi(a_{i-1,j-5})+8 &\geq& \left[\frac{3(j-5)-(i-1)}{2} \right] +8=\left[ \frac{3j-i+2}{2}\right]\geq\left[\frac{3j-i}{2} \right],\cr
\pi(a_{i-2,j-1})+1 &\geq& \left[\frac{3(j-1)-(i-2)}{2} \right] +1=\left[ \frac{3j-i+1}{2}\right]\geq\left[\frac{3j-i}{2} \right],\cr
\pi(a_{i-2,j-2})+3 &\geq& \left[\frac{3(j-2)-(i-2)}{2} \right] +3=\left[ \frac{3j-i+2}{2}\right]\geq\left[\frac{3j-i}{2} \right],\cr
\pi(a_{i-2,j-3})+5 &\geq& \left[\frac{3(j-3)-(i-2)}{2} \right] +5=\left[ \frac{3j-i+3}{2}\right]\geq\left[\frac{3j-i}{2} \right],\cr
\pi(a_{i-2,j-4})+6 &\geq& \left[\frac{3(j-4)-(i-2)}{2} \right] +6=\left[ \frac{3j-i+2}{2}\right]\geq\left[\frac{3j-i}{2} \right],\cr
\pi(a_{i-3,j-1})+1 &\geq& \left[\frac{3(j-1)-(i-3)}{2} \right] +1=\left[ \frac{3j-i+2}{2}\right]\geq\left[\frac{3j-i}{2} \right],\cr
\pi(a_{i-3,j-2})+3 &\geq& \left[\frac{3(j-2)-(i-3)}{2} \right] +3=\left[ \frac{3j-i+3}{2}\right]\geq\left[\frac{3j-i}{2} \right],\cr
\pi(a_{i-3,j-3})+4 &\geq& \left[\frac{3(j-3)-(i-3)}{2} \right] +4=\left[ \frac{3j-i+2}{2}\right]\geq\left[\frac{3j-i}{2} \right],\cr
\pi(a_{i-4,j-1})+1 &\geq& \left[\frac{3(j-1)-(i-4)}{2} \right] +1=\left[ \frac{3j-i+3}{2}\right]\geq\left[\frac{3j-i}{2} \right],\cr
\pi(a_{i-4,j-2})+2 &\geq& \left[\frac{3(j-2)-(i-4)}{2} \right] +2=\left[ \frac{3j-i+2}{2}\right]\geq\left[\frac{3j-i}{2} \right],\cr
\pi(a_{i-5,j-1}) &\geq& \left[\frac{3(j-1)-(i-5)}{2} \right] =\left[ \frac{3j-i+2}{2}\right]\geq\left[\frac{3j-i}{2} \right].
\end{eqnarray*}
From (\ref{star}), we have $\pi(a_{ij})\geq$ last term of every row above, so
\begin{equation*}
\pi(a_{ij})\geq \left[\frac{3j-i}{2} \right].
\end{equation*}
\end{proof}
\begin{lem}\label{lem4.2}
For all $i,j \geq 1$, we have
\begin{equation*}
\pi(b_{ij})\geq \left[\frac{5j-i-1}{6} \right].
\end{equation*}
\end{lem}
\begin{proof}

Since $b_{ij}$ are integers satisfying the same recurrence as $a_{ij}$, so (\ref{star}) holds for $b_{ij}$.
By checking the values for $1\leq i \leq 5$, the inequality holds in these cases. The values of $b_{ij}$ can be obtained from the data in the Appendix. For $i\geq 6$, it follows by induction on $i$, using (\ref{star}) replacing $a_{ij}$ by $b_{ij}$.
\end{proof}
\begin{lem}\label{lemma4.3}
\begin{eqnarray*}
&(1)_\alpha& \quad \pi\left(w_j^{(2\alpha+1)}\right) \geq \alpha +\left[\frac{j}{2} \right],\cr
&(2)_\alpha& \quad \pi\left(w_j^{(2\alpha+2)}\right) \geq \alpha +1+\left[\frac{j-1}{2} \right].
\end{eqnarray*}
\end{lem}
\begin{proof}
Since $W_1=(3,25,125,0,0,0\dots)$, so $(1)_0$ holds. Suppose $(1)_\alpha$ holds for some $\alpha\geq 0$.
$W_{2\alpha+2}= W_{2\alpha+1}a$, we have
$$
w_j^{(2\alpha+2)}= \sum_{i\geq 1}w_i^{(2\alpha+1)}a_{ij}.
$$
Hence
\begin{eqnarray*}
\pi\left(w_j^{(2\alpha+2)}\right)&\geq& \text{min}_{i\geq1}\left(\pi\left(w_i^{(2\alpha+1)}\right)+\pi(a_{ij})\right)\cr
&\geq& \text{min}_{i\geq1}\left(\alpha+\left[\frac{i}{2} \right]+\pi(a_{ij})\right).
\end{eqnarray*}by assumption.
We show $(2)_\alpha$ from this by showing that
\begin{equation}\label{ss}
\alpha+\left[\frac{i}{2} \right]+\pi(a_{ij}) \geq \alpha+1 + \left[ \frac{j-1}{2}\right]\, \text{for all}\, i,j \geq 1.
\end{equation}
Clearly (\ref{ss}) holds for $i\geq j+1$. If $i\leq j$, we have by Lemma \ref{lem4.1}
\begin{eqnarray*}
\pi(a_{ij}) &\geq& \left[\frac{3j-i}{2} \right] \geq \left[\frac{2j}{2} \right]\cr
&\geq& \left[ \frac{j+1}{2} \right]= 1+ \left[\frac{j-1}{2} \right].
\end{eqnarray*}
Thus $(1)_{\alpha}$ implies $(2)_{\alpha}$. Next suppose $(2)_{\alpha}$ holds for some $\alpha \geq 0$.
Since $W_{2\alpha +3}= W_{2\alpha+2}b$, we have
\begin{eqnarray*}
w_j^{(2\alpha+3)}= \sum_{i\geq 1}w_i^{(2\alpha+2)}b_{ij}.
\end{eqnarray*}
Hence
\begin{eqnarray*}
\pi\left(w_j^{(2\alpha+3)}\right)&\geq& \text{min}_{i\geq1}\left(\pi\left(w_i^{(2\alpha+2)}\right)+\pi(b_{ij})\right)\\
&\geq& \text{min}_{i\geq1}\left(\alpha+1+\left[\frac{i-1}{2} \right]+\pi(b_{ij})\right).
\end{eqnarray*}
From this we want to show that $(1)_{\alpha +1}$. This is equivalent to show that
\begin{equation}\label{starstar}
\left[ \frac{i-1}{2} \right] + \pi(b_{ij}) \geq \left[ \frac{j}{2} \right] \,\, \text{for all}\, i, j \geq 1
\end{equation}
Clearly (\ref{starstar}) holds for $i \geq j+1$. If $i\leq j$, then by Lemma \ref{lem4.2}
\begin{eqnarray*}
\pi(b_{ij}) \geq \left[\frac{5j-i-1}{6} \right] \geq \left[\frac{4j-1}{6} \right]
\end{eqnarray*}
\begin{eqnarray*}
\geq \left[ \frac{j}{2} \right], \,\,\text{for}\, j\geq 1.
\end{eqnarray*}
This completes the proof of Lemma \ref{lemma4.3}.
\end{proof}

Finally, $(2)_{\alpha}$ implies $W_{2\alpha+2}\equiv 0 \pmod{5^{\alpha+1}}$, this proves
Theorem \ref{thm1.4}.
\section{Appendix}
 In this section we give the values of $U(A^i)$ for $2\leq i \leq 5$ and $U(FA^i$) for $1\leq i \leq 5$.
 \begin{eqnarray*}
&&U(A^2)\cr
&=&56 A + 6675 A^2 + 247500 A^3 + 4862500 A^4 + 59062500 A^5 +482421875 A^6 \cr
 &+& 2695312500 A^7 + 10253906250 A^8 + 24414062500 A^9 +30517578125 A^{10},\cr
&&U(A^3)\cr
&=&33 A + 14850 A^2 + 1510625 A^3 + 70743750 A^4 + 1974140625 A^5\cr
 &+& 37035156250 A^6+ 499980468750 A^7 + 5047851562500 A^8 \cr
  &+& 38940429687500 A^9 + 231262207031250 A^{10}+1052856445312500 A^{11} \cr
   &+&3601074218750000 A^{12} + 8869171142578125 A^{13}\cr
  &+& 14305114746093750 A^{14} + 11920928955078125 A^{15}.\cr
&&U(A^4)\cr
&=& 427246093688 A + 4196167007060 A^2 + 19073490098000 A^3\cr
      &+& 47684078034375 A^4 + 19250500000 A^5 + 663960937500 A^6\cr
      &+& 16247343750000 A^7 + 298110351562500 A^8 + 4251562500000000 A^9\cr
      &+& 48268676757812500 A^{10} + 443010864257812500 A^{11}\cr
      &+& 3316589355468750000 A^{12} + 20321464538574218750 A^{13}\cr
      &+&101693630218505859375 A^{14} + 412178039550781250000 A^{15}\cr
      &+&1330971717834472656250 A^{16} + 3325939178466796875000 A^{17}\cr
      &+&6109476089477539062500 A^{18} + 7450580596923828125000 A^{19}\cr
      &+&4656612873077392578125 A^{20}.\cr
&&U(A^5)\cr
&=& A + 74768066403450 A^2 + 2229690556656875 A^3\cr
      &+& 32711030009900000 A^4 + 286102384413984375 A^5\cr
      &+& 1645093352222656250 A^6 + 6258692726748046875 A^7\cr
      &+& 14907182816894531250 A^8 + 18763117572021484375 A^9\cr
      &+& 2473277331542968750 A^{10} + 36498332977294921875 A^{11}\cr
      &+& 446328750610351562500 A^{12} + 4574938602447509765625 A^{13}\cr
      &+& 39611951828002929687500 A^{14} + 291057825088500976562500 A^{15}\cr
      &+& 1818010449409484863281250 A^{16} + 9642277657985687255859375 A^{17}\cr
      &+& 43237060308456420898437500 A^{18} + 162589177489280700683593750 A^{19}\cr
      &+& 505987554788589477539062500 A^{20} + 1276494003832340240478515625 A^{21}\cr
      &+&2526212483644485473632812500 A^{22} + 3710738383233547210693359375 A^{23}\cr
      &+&3637978807091712951660156250 A^{24} + 1818989403545856475830078125 A^{25}\cr
&&U(FA)\cr
&=& -7 A + 440 A^2 + 13875 A^3 + 206250 A^4 + 1750000 A^5 + 9375000 A^6 \cr
 & +& 29296875 A^7 +48828125 A^8\cr
&&U(FA^2)\cr
&=& -825 A^2 + 60000 A^3 + 3796875 A^4 + 99062500 A^5 + 1572265625 A^6 \cr
      &+& 17197265625 A^7+ 135742187500 A^8 + 787353515625 A^9 + 3326416015625 A^{10}\cr
      &+& 9918212890625 A^{11} + 19073486328125 A^{12} + 19073486328125 A^{13}\cr
&&U(FA^3)\cr
 &=& -805 A^2 - 90225 A^3 + 10131250 A^4 + 896734375 A^5 + 33180859375 A^6\cr
        &+& 766660156250 A^7 + 12565576171875 A^8 + 155086669921875 A^9\cr
        &+& 1488800048828125 A^{10} + 11305999755859375 A^{11}\cr
        &+& 68371582031250000 A^{12} + 328445434570312500 A^{13}\cr
        &+& 1238346099853515625 A^{14} + 3571510314941406250 A^{15}\cr
        &+& 7510185241699218750 A^{16} + 10430812835693359375 A^{17}\cr
        &+& 7450580596923828125 A^{18}\cr
&&U(FA^4)\cr
 &=& -354 A^2 - 233600 A^3 - 10470625 A^4 + 2053506250 A^5 +
 208768671875 A^6 \cr
        &+& 9661148437500 A^7 + 287924541015625 A^8 + 6220991210937500 A^9\cr
        &+&103437377929687500 A^{10} + 1370702392578125000 A^{11}\cr
        &+&14796306610107421875 A^{12} + 131903648376464843750 A^{13}\cr
        &+& 978770828247070312500 A^{14} + 6064968109130859375000 A^{15}\cr
        &+&31349420547485351562500 A^{16} + 134410262107849121093750 A^{17}\cr
        &+& 472672283649444580078125 A^{18} + 1337751746177673339843750 A^{19}\cr
        &+& 2953223884105682373046875 A^{20} + 4819594323635101318359375 A^{21}\cr
        &+&5238689482212066650390625 A^{22} + 2910383045673370361328125 A^{23}\cr
&&U(FA^5)\cr
&=& -67 A^2 - 215775 A^3 - 52613750 A^4 - 1147121875 A^5 +
 459985703125 A^6\cr
       &+& 49379214843750 A^7 + 2650373671875000 A^8 + 94741486816406250 A^9\cr
       &+& 2506321456298828125 A^{10} + 51876175903320312500 A^{11}\cr& +&
 869309848022460937500 A^{12}
       +  12063823905944824218750 A^{13}\cr& +& 140814201641082763671875 A^{14}
       +1397377154827117919921875 A^{15}\cr& +& 11873932600021362304687500 A^{16}
       + 86767370402812957763671875 A^{17}\cr& +& 546239309012889862060546875 A^{18}
       +2960891649127006530761718750 A^{19}\cr& +& 13776480220258235931396484375 A^{20}
      +  54694735445082187652587890625 A^{21}\cr& +&183524098247289657592773437500 A^{22}
      + 512963742949068546295166015625 A^{23}\cr& +&
 1168518792837858200073242187500 A^{24}
     + 2096930984407663345336914062500 A^{25}\cr& +&
 2801243681460618972778320312500 A^{26}
     + 2501110429875552654266357421875 A^{27}\cr& +&
 1136868377216160297393798828125 A^{28}
 \end{eqnarray*}




\end{document}